\input ./style/arxiv-vmsta.cfg
\documentclass[numbers,compress,v1.0.1]{vmsta}

\volume{2}
\issue{4}
\pubyear{2015}
\firstpage{327}
\lastpage{341}
\doi{10.15559/15-VMSTA34}

\startlocaldefs
\newcommand{\longrightarrowfdd}{\stackrel{f.d.d.}{\longrightarrow}}

\newtheorem{prop}{Proposition}
\newtheorem{thm}{Theorem}
\newtheorem{lemma}{Lemma}

\theoremstyle{definition}
\newtheorem{rem}{Remark}
\newtheorem{defin}{Definition}

\allowdisplaybreaks
\endlocaldefs

\begin{document}
\begin{frontmatter}

\title{Tempered Hermite process}

\author{\inits{F.}\fnm{Farzad}\snm{Sabzikar}}\email{sabzikar@iastate.edu}

\address{Department of Statistics, Iowa State University, Ames, IA
50010, USA}

\markboth{F. Sabzikar}{Tempered Hermite process}

\begin{abstract}
A tempered Hermite process modifies the power law kernel in the
time domain representation of a Hermite process by multiplying an exponential
tempering factor $\lambda>0$ such that the process is well defined for
Hurst parameter $H>\frac{1}{2}$.
A~tempered Hermite process is the weak convergence limit of a certain
discrete chaos process.
\end{abstract}


\received{9 July 2015}
\revised{7 September 2015}
\accepted{11 September 2015}
\publishedonline{25 September 2015}
\end{frontmatter}

\section{Introduction}
The Hermite processes of order $k=1,2,\ldots$ are defined as multiple
Wiener--It\^{o} integrals
%
\begin{equation}
\label{eq:THPdefnorderk} Z_{H}^{k}(t):={\int_{\mathbb{R}^k}^{'}
\int_{0}^{t} \Biggl(\prod
^{k}_{i=1}(s-y_i)_{+}^{d-1}
\Biggr)\,ds\,B(dy_1)\ldots B(dy_k)},
\end{equation}
where $d=\frac{1}{2}-\frac{1-H}{k}\in(\frac{1}{2}-\frac
{1}{2k},\frac{1}{2})$ and $\frac{1}{2}<H<1$ (the prime $'$ on the
integral sign shows that one does not integrate on the diagonals
$x_{i}=x_{j}$, $i\neq j$). They are self-similar processes with
stationary increments (see \cite{Dobrushinmajor,Taqqu}).

In this paper, we introduce a new class of stochastic processes, which
we call tempered Hermite processes. Tempered Hermite processes modify
the kernel of $Z_{H}^{k}$ by multiplying an exponential
tempering factor $\lambda>0$ such that they are well defined for Hurst
parameter $H>\frac{1}{2}$. Tempered Hermite processes are not
self-similar processes, but they have a scaling property, involving
both the time scale and the tempering parameter. The scaling property
enable us to show that the tempered Hermite processes are the weak
convergence limits of certain discrete chaos processes.\vadjust{\eject}

The paper is organized as follows. In Section \ref{sec2}, we define
tempered Hermite processes and derive some their basic properties. In
Section~\ref{sec3}, we present our main result on the weak convergence
to tempered Hermite processes.
\section{Tempered Hermite process}\label{sec2}
Let $B=\{B(t), t\in\mathbb{R}\}$ be a real-valued Brownian motion on
the real line, a~process with stationary independent increments such
that $B(t)$ has a Gaussian distribution with mean zero and variance
$|t|$ for all $t\in\mathbb{R}$. Then the Wiener--It\^{o} integrals
\begin{equation*}
I_{k}(f):=\int_{\mathbb{R}^k}^{'}
f(x_1,\ldots,x_k)B(dx_1)\ldots
B(dx_k)
\end{equation*}
are defined for all functions $f\in L^{2}(\mathbb{R}^k)$.
The prime $'$ on the integral sign shows that one does not integrate on
the diagonals $x_{i}=x_{j}$, $i\neq j$. See, for example, \cite
[Chapter 4]{koul}.
\begin{defin}
Let $H>\frac{1}{2}$ and $\lambda>0$. The process
%
\begin{equation}
\label{eq:THPdefn} Z_{H,\lambda}^{k}(t):={\int_{\mathbb{R}^k}^{'}
\int_{0}^{t}\prod_{i=1}^{k}
\bigl({(s-y_i)_{+}^{d-1}}e^{-\lambda(s-y_i)_{+}}
\bigr)\,ds\,B(dy_1)\ldots B(dy_k)},
\end{equation}
where $(x)_{+}=xI(x>0)$ and $d=\frac{1}{2}-\frac{1-H}{k}\in(\frac
{1}{2}-\frac{1}{2k},\infty)$, is called a tempered Hermite process of
order k.
\end{defin}
The next lemma shows that $Z_{H,\lambda}^{k}(t)$, given by \eqref
{eq:THPdefn}, is well defined for any $t\geq0$.
\begin{lemma}\label{lem:g squar integrable}
The function
%
\begin{equation}
\label{eq:integrand} h_{t}(y_1,\dots,y_k):=\int
_{0}^{t}\prod_{i=1}^{k}(s-y_i)_{+}^{d-1}e^{-\lambda(s-y_i)_{+}}\,ds
\end{equation}
is well defined in $L^{2}(\mathbb{R}^k)$ for any $H>\frac{1}{2}$ and
$\lambda>0$.
\end{lemma}
\begin{proof}
To show that $h_{t}(y_1,\ldots,y_k)$ is in $L^{2}(\mathbb{R}^k)$, we write
\begin{align}
\label{eq:variancegeneral} %
&\int_{\mathbb{R}^k}h_{t}(y_1,\ldots,y_k)^{2}\,dy_1\ldots dy_k\notag\\
&\quad =\int_{\mathbb{R}^k} \Biggl[\int_{0}^{t}\int_{0}^{t}\prod_{i=1}^{k}(s_1-y_i)_{+}^{d-1}e^{-\lambda(s_1-y_i)_{+}}(s_2-y_i)_{+}^{d-1}\notag\\
&\qquad \times e^{-\lambda(s_2-y_i)_{+}}\,ds_1\,ds_2 \Biggr]\,dy_1\ldots dy_k\notag\\
&\quad =2\int_{0}^{t}ds_1\int_{s_1}^{t}ds_2 \Biggl[\int_{\mathbb{R}^k}\prod_{i=1}^{k}(s_1-y_i)_{+}^{d-1}e^{-\lambda(s_1-y_i)_{+}}(s_2-y_i)_{+}^{d-1}\notag\\
&\qquad \times e^{-\lambda(s_2-y_i)_{+}}\,dy_1\ldots dy_k \Biggr]\notag\\
&\quad =2\int_{0}^{t}ds\int_{0}^{t-s}\xch{du\,\Biggl[}{du}\int_{\mathbb{R}^k_{+}}\prod_{i=1}^{k}w_{i}^{d-1}e^{-\lambda w_i}(w_i+u)^{d-1}e^{-\lambda(w_i+u)}\,dw_1\ldots dw_k \Biggr]\notag\\
&\qquad(s=s_1, u=s_2-s_1,w_i=s_1-y_i)\notag\\
&\quad =2\int_{0}^{t}ds\int_{0}^{t-s}e^{-\lambda uk}\,du\,\biggl[\int_{\mathbb{R}_{+}}w^{d-1}(w+u)^{d-1}e^{-2\lambda w}\,dw \biggr]^{k}\notag\\
&\quad =2\int_{0}^{t}ds\int_{0}^{t-s}e^{-\lambda uk}u^{k(2d-1)}\,du\, \biggl[\int_{\mathbb{R}_{+}}x^{d-1}(x+1)^{d-1}e^{-2\lambda ux}\,dx \biggr]^{k}\notag\\
&\quad =2\int_{0}^{t}ds\int_{0}^{t-s}e^{-\lambda uk}u^{k(2d-1)}\,du \biggl[\frac{\varGamma(d)}{\sqrt{\pi}} \biggl(\frac{1}{2\lambda u}\biggr)^{d-\frac{1}{2}}e^{\lambda u}K_{\frac{1}{2}-d}(\lambda u)\biggr]^{k}\notag\\
&\quad =2 \biggl[\frac{\varGamma(d)}{\sqrt{\pi}(2\lambda)^{d-\frac{1}{2}}} \biggr]^{k}\int_{0}^{t}ds\int_{0}^{t-s} \bigl[u^{d-\frac{1}{2}}K_{\frac{1}{2}-d}(\lambda u) \bigr]^{k}\,du\notag\\
&\quad =2 \biggl[\frac{\varGamma(d)}{\sqrt{\pi}{2}^{d-\frac{1}{2}}{\lambda}^{2d-1}} \biggr]^{k}\int_{0}^{t}ds\int_{0}^{\lambda(t-s)} \bigl[z^{d-\frac{1}{2}}K_{\frac{1}{2}-d}(z)\bigr]^{k}\,dz,
\end{align}
where we applied the standard integral formula \cite[p.~344]{Gradshteyn}
%
\begin{equation}
\label{eq:standard integral formula} \int_{0}^{\infty}x^{\nu-1}(x+
\beta)^{\nu-1}e^{-\mu x}\,dx=\frac
{1}{\sqrt{\pi}} \biggl(
\frac{\beta}{\mu} \biggr)^{\nu-\frac
{1}{2}}e^{\frac{\beta\mu}{2}} \varGamma(
\nu)K_{\frac{1}{2}-\nu}\biggl(\frac{\beta\mu}{2}\biggr)
\end{equation}
for $|\arg\beta|<\pi$, Re $\mu>0$, Re $\nu>0$. Here $K_{\nu}(x)$
is the modified Bessel function of the second kind (see \cite[Chapter
9]{abramowitz}). Next, we need to show that
\begin{equation*}
\int_{0}^{t}ds\int_{0}^{\lambda(t-s)}
\bigl[z^{d-\frac
{1}{2}}K_{\frac{1}{2}-d}(z) \bigr]^{k}\,dz
\end{equation*}
is finite for $d>\frac{1}{2}-\frac{1}{2k}$ (equivalently, for
$H>\frac{1}{2}$). First, suppose $\frac{1}{2}-\frac{1}{2k}<d<\frac
{1}{2}$ (or $\frac{1}{2}<H<1$). Since $k_{\nu}(z)<z^{-\nu}2^{\nu
-1}\varGamma(\nu)$ for $z>0$ (Theorem~3.1 in \cite{Gaunt}), we have
\begin{align}
\label{eq:variancelesshalf} %
&\int_{0}^{t}ds\int_{0}^{\lambda(t-s)} \bigl[z^{d-\frac{1}{2}}K_{\frac{1}{2}-d}(z)\bigr]^{k}\,dz\notag\\
&\quad \leq \biggl[2^{-(\frac{1}{2}+d)}\varGamma\biggl(\frac{1}{2}-d\biggr) \biggr]^{k}\int_{0}^{t}ds\int_{0}^{\lambda(t-s)}z^{k(2d-1)}\,dz\notag\\
&\quad =\frac{ [\lambda^{2d-1}2^{-(\frac{1}{2}+d)}\varGamma(\frac{1}{2}-d) ]^{k}\lambda}{(k(2d-1)+1)(k(2d-1)+2)}t^{2kd-k+2},
\end{align}
which is finite, and, consequently, from \eqref{eq:variancegeneral}
and \eqref{eq:variancelesshalf} we get
\begin{equation*}
\int_{\mathbb{R}^k}h_{t}(y_1,
\ldots,y_k)^{2}\,dy_1\ldots dy_k<
\frac{2\lambda [\frac{\varGamma(d)\varGamma(\frac{1}{2}-d)}{\sqrt
{\pi}{2}^{2d}} ]^{k}}{(k(2d-1)+1)(k(2d-1)+2)}t^{2kd-k+2}
\end{equation*}
for $\frac{1}{2}-\frac{1}{2k}<d<\frac{1}{2}$. Next, suppose $d>\frac
{1}{2}$ (equivalently $H>1$). In this case,\vadjust{\eject}
\begin{align}
\label{eq:variancegreatherhalf} %
&\int_{0}^{t}ds\int_{0}^{\lambda(t-s)} \bigl[z^{d-\frac{1}{2}}K_{\frac{1}{2}-d}(z)\bigr]^{k}\,dz= \int_{0}^{t}ds\int_{0}^{\lambda(t-s)} \bigl[z^{d-\frac{1}{2}}K_{d-\frac{1}{2}}(z)\bigr]^{k}\,dz\notag\\
&\quad\leq\int_{0}^{t}ds\int_{0}^{\lambda(t-s)}\biggl[2^{d-\frac{3}{2}}\varGamma\biggl(d-\frac{1}{2}\biggr)\biggr]^{k}\,dz\leq\frac{\lambda [2^{d-\frac{3}{2}}\varGamma(d-\frac{1}{2}) ]^{k}}{2}t^{2},
\end{align}
where we applied the fact that $K_{\nu}(z)=K_{-\nu}(z)$ and $k_{\nu
}(z)<z^{-\nu}2^{\nu-1}\varGamma(\nu)$ for $z>0$. Hence, from \eqref
{eq:variancegeneral} and \eqref{eq:variancegreatherhalf} it follows that
\begin{equation*}
\int_{\mathbb{R}^k}h_{t}(y_1,
\ldots,y_k)^{2}\, dy_1\ldots dy_k<
\biggl[\frac{\varGamma(d)\varGamma(d-\frac{1}{2})}{2\sqrt{\pi}{\lambda
}^{2d-1}} \biggr]^{k}\lambda\, t^{2}
\end{equation*}
for $d>\frac{1}{2}$. Finally, let $d=\frac{1}{2}$ (equivalently,
$H=1$). Consider
\begin{align}
\label{eq:varianceequalshalf1} %
\int_{0}^{\lambda(t-s)}\bigl(K_{0}(z) \bigr)^{k}\,dz&=\int_{0}^{\eta}\bigl(K_{0}(z) \bigr)^{k}\,dz+\int_{\eta}^{\lambda(t-s)}\bigl(K_{0}(z) \bigr)^{k}\,dz\notag\\
&:=I_{1}+I_{2}.
\end{align}
Since $K_{0}(z)\sim-\log(z)$ as $z\to0$ (see \cite[Eq.
9.6.8]{abramowitz}, we have
\begin{align}
\label{eq:varianceequalshalf2} %
I_{1}&=\int_{0}^{\eta} \biggl(\frac{K_{0}(z)}{\log(z)}\log(z)\biggr)^{k}\,dz\leq(1+\epsilon)^{k}\int_{0}^{\eta}\bigl(-\log(z)\bigr)^{k}\,dz\notag\\
&=(1+\epsilon)^{k}\int_{-\log(\eta)}^{+\infty}w^{k}e^{-w}\,dz\leq (1+\epsilon)^{k}\int_{0}^{+\infty}w^{k}e^{-w}\,dz\notag\\
&=(1+\epsilon)^{k}\varGamma(k+1).
\end{align}
Now, we find an upper bound for $I_{2}$. It can be shown that $\frac
{K_{\nu}(x)}{K_{\nu}(y)}>e^{y-x}$ for $0<x<y$ and an arbitrary real
number $\nu$ (see \cite{Baricz}). Therefore,
\begin{align}
\label{eq:varianceequalshalf3} %
I_{2}&=\int_{\eta}^{\lambda(t-s)} \bigl(K_{0}(z)\bigr)^{k}\,dz<\int_{\eta}^{\lambda(t-s)}\bigl(K_{0}(\eta)e^{\eta-z} \bigr)^{k}\,dz\,
\bigl[K_{0}(\eta)e^{\eta} \bigr]^{k}\bigl(\lambda(t-s)-\eta\bigr).
\end{align}
From \eqref{eq:varianceequalshalf1}, \eqref{eq:varianceequalshalf2},
and \eqref{eq:varianceequalshalf3} we can see that
\begin{equation*}
\int_{0}^{\lambda(t-s)} \bigl(K_{0}(z)
\bigr)^{k}\,dz< (1+\epsilon)^{k}\varGamma(k+1)+
\bigl[K_{0}(\eta)e^{\eta} \bigr]^{k}\bigl(
\lambda(t-s)-\eta\bigr)
\end{equation*}
and hence
%
\begin{equation}
\int_{0}^{t}ds\int_{0}^{\lambda(t-s)}
\bigl[K_{\frac{1}{2}-d}(z) \bigr]^{k}\,dz <\bigl((1+
\epsilon)^{k}\varGamma(k+1)\bigr)t\,+\, \bigl[K_{0}(
\eta)e^{\eta} \bigr]^{k} \biggl(\frac{\lambda t^2}{2}-\eta t
\biggr)
\end{equation}
for $\epsilon>0$, and this shows that
\begin{align*}
&\int_{\mathbb{R}^k}h_{t}(y_1,\ldots,y_k)^{2}\,dy_1\ldots dy_k\\
&\quad < 2 \biggl[\frac{\varGamma(d)}{\sqrt{\pi}{\lambda}^{2d-1}}2^{d-\frac{1}{2}}\biggr]^{k} \biggl[\bigl((1+\epsilon)^{k}\varGamma(k+1)\bigr)t+
\bigl[K_{0}(\eta)e^{\eta
} \bigr]^{k} \biggl(
\frac{\lambda t^2}{2}-\eta t \biggr) \biggr]
\end{align*}
for $d=\frac{1}{2}$ ($H=1$), which completes the proof.
\end{proof}
The next result shows that although a tempered Hermite process is not
a~self-similar process, it does have a nice scaling property. Here the
symbol $\triangleq$ indicates the equivalence of finite-dimensional
distributions.
\begin{prop}\label{prop:sssi}
The process $Z^{k}_{H,\lambda}$ given by \eqref{eq:THPdefn}
has stationary increments such that
%
\begin{equation}
\label{eq:scalingTHP} \bigl\{Z_{H,\lambda}^{k}(ct) \bigr
\}_{t\in\mathbb{R}}{\triangleq } \bigl\{ c^{H}Z_{H,c\lambda}^{k}(t)
\bigr\}_{t\in\mathbb{R}}
\end{equation}
for any scale factor $c>0$.
\end{prop}
\begin{proof}
Since $B(dy)$ has the control measure $m(dy)=\sigma^2 \,dy$, the random
measure $B(c\,dy)$ has the control measure $c^{1/2}\sigma^2\, dy$. Given
$t_j$, $j=1,\dots,n$, by the change of variables ${s=cs'}$ and
${y_i=cy^{'}_{i}}$ for $i=1\ldots,k$ we have
\begin{align*}
Z_{H,\lambda}^{k}(ct_j)&\,{=}\,{\int_{\mathbb{R}^k}\!\int_{0}^{ct_j} \Biggl(\prod_{i=1}^{k}(s\,{-}\,y_i)_{+}^{-(\frac{1}{2}+\frac{1-H}{k})}e^{-\lambda(s-y_i)_{+}}\Biggr)\,ds\,B(dy_1)\ldots B(dy_k)}\\
&\,{=}\,{\int_{\mathbb{R}^k}\!\int_{0}^{t_j}\Biggl(\prod_{i=1}^{k}\bigl(cs'{-}\,cy^{'}_i\bigr)_{+}^{-(\frac{1}{2}+\frac{1-H}{k})}e^{-\lambda(cs'-cy^{'}_{i})_{+}} \!\Biggr) c\,ds'\,B\bigl(c dy'_1\bigr)\ldots B\bigl(c dy'_k\bigr)}\\
&{\,\triangleq}\, c^{H}{\int_{\mathbb{R}}\!\int_{0}^{t_j} \bigl({\bigl(s'\,{-}\,y'\bigr)_{+}^{-(\frac{1}{2}+\frac{1-H}{k})}}e^{-\lambda c(s'-y')_{+}} \bigr)\,ds'\,B\bigl(dy'_1\bigr)\ldots B\bigl(dy'_k\bigr)}\\
&\,{=}\,c^{H}Z_{H,c\lambda}^{k}(t_j),
\end{align*}
so that \eqref{eq:scalingTHP} holds. Suppose now that $s_j<t_j$ and
change the variables $x=x'+s_j$, $y_i=s_j+y'_i$ (for $j=1,\ldots,n$)
to get
\begin{align*}
& \bigl(Z_{H,\lambda}^{k}(t_j)-Z_{H,\lambda}^{k}(s_j)\bigr)\\
&\quad ={\int_{\mathbb{R}^k}\int_{s_j}^{t_j}\Biggl(\prod^{k}_{i=1}{(x-y_i)_{+}^{-(\frac{1}{2}+\frac{1-H}{k})}}e^{-\lambda(x-y_i)_{+}}\Biggr)\,dx\, B(dy_1) \ldots B(dy_k)}\\
&\quad ={\int_{\mathbb{R}^k}\int_{0}^{t_j-s_j}\Biggl(\prod^{k}_{i=1}{\bigl(x'+s_j-y_i\bigr)_{+}^{-(\frac{1}{2}+\frac{1-H}{k})}}e^{-\lambda(x'+s_j-y_i)_{+}} \Biggr)\,dx'\,B(dy_1)\ldots B(dy_k)}\\
&\quad \triangleq{\int_{\mathbb{R}^k}\int_{0}^{t_j-s_j}\Biggl(\prod^{k}_{i=1}{\bigl(x'_i-y'_i\bigr)_{+}^{-(\frac{1}{2}+\frac{1-H}{k})}}e^{-\lambda(x'_i-y'_i)_{+}} \Biggr)\,dx'\, B\bigl(dy'_1\bigr)\ldots B\bigl(dy'_k\bigr)}\\
&\quad =Z_{H,\lambda}^{k}(t_j-s_j),
\end{align*}
which shows that a tempered Hermite process of order $k$ has stationary
increments.
\end{proof}
As a consequence of Lemma~\ref{lem:g squar integrable} and Proposition~\ref{prop:sssi}, we get the following:
\begin{prop}
The stochastic process $Z^{k}_{H,\lambda}(t)$ has a continuous
version.
\end{prop}
\begin{proof}
According to the proof of Lemma~\ref{lem:g squar integrable},
%
\begin{equation}
\mathbb{E}\big|Z^{k}_{H,\lambda}(t)-Z^{k}_{H,\lambda}(s)\big|^{2}
\leq %
\begin{cases} c_{1}|t-s|^{2H} & \frac{1}{2}<H<1,\\ c_{2}|t-s|^{2} &
H>1,
\end{cases} %
\end{equation}
where $c_1$ and $c_2$ are some constants. Kolmogorov's continuity
criterion states that a stochastic process $X(t)$ has a continuous
version if there exist some positive constants $p$, $\beta$, and $c$
such that
%
\begin{equation}
\label{eq:kolmogorov} \mathbb{E}\big|X(t)-X(s)\big|^{p}\leq c|t-s|^{1+\beta}.
\end{equation}
Apply \eqref{eq:kolmogorov} for tempered Hermite process
$Z^{k}_{H,\lambda}(t)$ by taking $p=1$, $\beta=\min\{1,\break 2H-1\}$, and
$c=\min\{c_1,c_2\}$ to get the desired result.
\end{proof}
We next compute the covariance function of $Z^{k}_{H,\lambda}(t)$.
Unlike Hermite processes, the covariance function of a tempered Hermite
process is different for different $k\geq1$.
\begin{prop}\label{prop:THPcovariance}
The process $Z^{k}_{H,\lambda}$ given by \eqref{eq:THPdefn}
has the covariance function
\begin{equation*}
R(t,s)=2 \biggl[\frac{\varGamma(d)}{\sqrt{\pi}(2\lambda)^{d-\frac
{1}{2}}} \biggr]^{k}\int_{0}^{t}
\int_{0}^{s} \bigl[|u-v|^{d-\frac
{1}{2}}K_{\frac{1}{2}-d}\big(
\lambda|u-v|\big) \bigr]^{k}\,dv\,du
\end{equation*}
for $\lambda>0$ and $d>\frac{1}{2}-\frac{1}{2k}$
(equivalently, $H>\frac{1}{2}$).
\end{prop}

\begin{proof}
By applying the Fubini theorem and the isometry of multiple Wiener--It\^
{o} integrals we have
\begin{align*}
R(t,s)&=2\int_{\mathbb{R}^k} \Biggl(\int_{0}^{t}\int_{0}^{s}\prod^{k}_{i=1}(u-y_i)_{+}^{d-1}(v-y_i)_{+}^{d-1}\\[4pt]
&\quad \times e^{-\lambda(u-y_i)_{+}}e^{-\lambda(v-y_i)_{+}}\,dv\,du \Biggr)\,dy_1\ldots dy_k\\[4pt]
&=2\int_{0}^{t}\int_{0}^{s}\int_{\mathbb{R}^k} \Biggl[\prod^{k}_{i=1}(u-y_i)_{+}^{d-1}(v-y_i)_{+}^{d-1}\\[4pt]
&\quad\times e^{-\lambda(u-y_i)_{+}}e^{-\lambda(v-y_i)_{+}}\,dy_1\ldots dy_k \Biggr]\,dv\,du\\[4pt]
&=2\int_{0}^{t}\int_{0}^{s} \biggl[\int_{\mathbb{R}}(u-y)_{+}^{d-1}(v-y)_{+}^{d-1}e^{-\lambda(u-y)_{+}}e^{-\lambda(v-y)_{+}}\,dy\biggr]^{k}\,dv\,du\\[4pt]
&=2\int_{0}^{t}\int_{0}^{s} \Biggl[\int_{-\infty}^{\min(u,v)}(u-y)^{d-1}(v-y)^{d-1}e^{-\lambda(u-y)}e^{-\lambda(v-y)}\,dy\Biggr]^{k}\,dv\,du\\[4pt]
&=2\int_{0}^{t}\int_{0}^{s} \Biggl[\int_{0}^{+\infty}w^{d-1}\big(|u-v|+w\big)^{d-1}e^{-\lambda w}e^{-\lambda(|u-v|+w)}\,dw\Biggr]^{k}\,dv\,du\\[4pt]
&=2\int_{0}^{t}\int_{0}^{s}e^{-\lambda k|u-v|}\, |u-v|^{k(2d-1)}\\[4pt]
&\quad\times \Biggl[\int_{0}^{+\infty}x^{-(\frac{1}{2}+\frac{1-H}{k})}(x+1)^{-(\frac{1}{2}+\frac{1-H}{k})}e^{-2\lambda|u-v|x}\,dx \Biggr]^{k}\,dv\,du\\[4pt]
&=2\int_{0}^{t}\int_{0}^{s}e^{-\lambda k|u-v|}\, |u-v|^{k(2d-1)}\\
&\quad\times \biggl[\frac{\varGamma(d)}{\sqrt{\pi}} \biggl(\frac{1}{2\lambda|u-v|} \biggr)^{d-\frac{1}{2}}e^{\lambda|u-v|} K_{\frac{1}{2}-d}\big(\lambda|u-v|\big) \biggr]^{k}\,dv\,du\\
&=2 \biggl[\frac{\varGamma(d)}{\sqrt{\pi}(2\lambda)^{d-\frac{1}{2}}} \biggr]^{k}\int_{0}^{t}\int_{0}^{s}\bigl[|u-v|^{d-\frac{1}{2}}K_{\frac{1}{2}-d}\big(\lambda|u-v|\big) \bigr]^{k}\,dv\,du
\end{align*}
for any $H>\frac{1}{2}$ and $\lambda>0$, and hence we get the desired result.
\end{proof}

Let $\widehat B_1$ and $\widehat B_2$ be independent Gaussian random
measures with $\widehat B_{1}(A)=\widehat B_{1}(-A)$, $\widehat
B_{2}(A)=-\widehat B_{2}(-A)$, and $\mathbb{E}[(\widehat B_{i}(A))^2]=m(A)/2$,
where $m(dx)=\sigma^2 \,dx$, and define the complex-valued Gaussian
random measure $\widehat B=\widehat B_{1}+i\widehat B_{2}$.
\begin{prop}\label{prop:THPdefharmo}
Let $H>\frac{1}{2}$ and $\lambda>0$. The process
$Z^{k}_{H,\lambda}$ given by \eqref{eq:THPdefn} has the spectral
domain representation
%
\begin{align}
\label{eq:THPdefharmo} Z^{k}_{H,\lambda}(t)&=C_{H,k}\int
_{\mathbb{R}^k}^{''} \frac
{e^{it(\omega
_1+\cdots+\omega_k)}-1}{i(\omega_1+\cdots+\omega_k)} \notag\\
&\quad\times\prod
_{j=1}^{k}(\lambda+i\omega_j)^{- (\frac{1}{2}-\frac
{1-H}{k} )}
\widehat{B}(d\omega_1)\ldots\widehat{B}(d\omega_k),
\end{align}
where $\widehat{B}(\cdot)$ is a complex-valued Gaussian random
measure, and $C_{H,k}= (\frac{\varGamma(\frac{1}{2}-\frac
{1-H}{k})}{\sqrt{2\pi}} )^{k}$ is a constant depending on $H$ and
$k$. The double prime ${''}$ on the integral indicates that one does
not integrate on the diagonals $\omega_i=\omega_j$, $i\neq j$.
\end{prop}
\begin{proof}
We first observe that
%
\begin{equation}
\label{eq:A definitionorderk} h_{t}(y_1,\ldots,y_k)=\int
_{0}^{t}\prod_{j=1}^{k}(s-y_j)_{+}^{d-1}e^{-\lambda(s-y_j)_{+}}
\,ds
\end{equation}
has the Fourier transform
\begin{align*}
&\widehat{h_{t}}(\omega_1,\ldots,\omega_k)\\
&\quad = \frac{1}{(2\pi)^\frac{k}{2}}\int_{\mathbb{R}^k}e^{i\sum_{j=1}^{k}\omega_j y_j}\int_{0}^{t}\prod_{j=1}^{k}(s-y_j)^{d-1}_{+}e^{-\lambda(s-y_j)_{+}}\,ds \,dy_1\ldots dy_k\\
&\quad = \frac{1}{(2\pi)^\frac{k}{2}}\int_{\mathbb{R}^k}\int_{0}^{t}e^{i\sum_{j=1}^{k}\omega_j (s-u_j)}\prod_{j=1}^{k}(u_j)^{d-1}_{+}e^{-\lambda(u_j)_{+}}\,ds \,du_1\ldots du_k\\
&\quad = \frac{1}{(2\pi)^\frac{k}{2}}\int_{0}^{t}\int_{\mathbb{R}^k}e^{is\sum_{j=1}^{k}\omega_j} \prod_{j=1}^{k}(u_j)^{d-1}_{+}e^{-(\lambda+i\omega_j)u_j}\,du_1\ldots du_k\, ds\\
&\quad = \biggl[\frac{\varGamma(d)}{\sqrt{2\pi}} \biggr]^{k}\frac{e^{it(\omega_1+\cdots+\omega_k)}-1}{i(\omega_1+\cdots+\omega_k)} \prod_{j=1}^{k}(\lambda+i\omega_j)^{-d}\\
&\quad = \biggl[\frac{\varGamma(\frac{1}{2}-\frac{1-H}{k})}{\sqrt{2\pi}} \biggr]^{k}\frac{e^{it(\omega_1+\cdots+\omega_k)}-1}{i(\omega_1+\cdots+\omega_k)} \prod_{j=1}^{k}(\lambda+i\omega_j)^{- (\frac{1}{2}-\frac{1-H}{k} )},
\end{align*}
using the well-known formula for the characteristic function of the
gamma density. Then \eqref{eq:THPdefn}, together with Proposition~9.3.1 in \cite{peccatitaqqu}, implies that
\begin{equation*}
\begin{split} Z^{k}_{H,\lambda}(t)&=\int
_{\mathbb{R}^k}^{'}h_{t}(y_1,\ldots
,y_k)B(dy_1)\ldots B(dy_k)
\\
&\triangleq\int_{\mathbb{R}^k}^{''}\widehat{h_{t}}(
\omega_1,\ldots ,\omega_k)\widehat{B}(d
\omega_1)\ldots\widehat{B}(d\omega_k)
\\
&=C_{H,k}\int_{{\mathbb{R}}^k}^{''}
\frac{e^{it(\omega_1+\cdots
+\omega_k)}-1}{i(\omega_1+\cdots+\omega_k)} \prod_{j=1}^{k}(\lambda+i
\omega_j)^{- (\frac{1}{2}-\frac
{1-H}{k} )} \widehat{B}(d\omega_1)\ldots
\widehat{B}(d\omega_k), \end{split} %
\end{equation*}
which is equivalent to \eqref{eq:THPdefharmo}.
\end{proof}

\section{Limit theorem}\label{sec3}
In this section, we show that the process $Z^{k}_{H,\lambda}(t)$ is
the weak convergence limit of a certain discrete chaos process. Our
approach follows the seminal work of Bai and Taqqu \cite{BaiTaqqu}.
When $k=1$ and $\lambda>0$, the discrete process $Y^{\lambda,k}(n)$,
\eqref{eq:chaostempereddiscrete}, is a time series that is useful to
model turbulence \cite{Meerschaertsabzikarkumarzeleki,TFC}. When $k=1$
and $\lambda=0$, Davydov \cite{Dovydov} (see also Giraitis et
al.~\cite[p. 276]{koul} and Whitt \cite[Theorem 4.6.1]{Whitt})
established the corresponding invariance principle for $Y^{\lambda
,k}(n)$, where the limit involves a fractional Brownian motion. When
$k>1$ and $\lambda=0$, Taqqu \cite{Taqqu} showed that the weak
convergence limit of $Y^{\lambda,k}(n)$ is the Hermite process~\eqref
{eq:THPdefnorderk}.

The following proposition gives a powerful tool for proving the result
of this section.
\begin{prop}\label{prop:koultaqquptoposition}
Let
%
\begin{equation}
\label{eq:Qdefinition} Q_{k}(g_N):=\sum
_{(j_1,\ldots,j_k)\in\mathbb
{Z}^{k}}^{'}g_{N}(j_1,
\ldots,j_k)\varepsilon_{j_1}\ldots\varepsilon_{j_k}
\end{equation}
for $N=1,2,\ldots$, where $g_{N}\in L^{2}(\mathbb{Z}^{k})$ for $k\geq
1$, and $\{\varepsilon_n\}$ is an i.i.d. sequence with mean zero and
variance 1. Assume that, for some $f\in L^{2}(\mathbb{R}^{k})$,
\begin{equation*}
\int_{\mathbb{R}^{k}}\big|\tilde{g}_N(u_1,
\ldots,u_k)-f(u_1,\ldots ,u_k)\big|^{2}\,du_1
\ldots du_{k}\to0,\quad \text{as } N\to\infty,
\end{equation*}
where
\begin{equation*}
\tilde{g}_N(u_1,\ldots,u_k):=N^{\frac{k}{2}}g_{N}
\bigl([u_1 N]+c_1,\ldots ,[u_k
N]+c_{k}\bigr),\quad(c_1,\ldots,c_k)\in
\mathbb{Z}^{k}.
\end{equation*}
Then
\begin{equation*}
Q_{k}(g_N)\longrightarrowfdd\int_{\mathbb{R}^{k}}f(u_1, \ldots ,u_k)B(du_1)\ldots B(du_k)
\end{equation*}
as $N\to\infty$.
\end{prop}
\begin{proof}
See Proposition~4.1 in \cite{BaiTaqqu} and also Corollary~4.7.1 in
\cite{koul}.\vadjust{\eject}
\end{proof}

Define the discrete chaos process
%
\begin{equation}
\label{eq:chaostempereddiscrete} Y^{\lambda,k}(n):=\sum_{(i_1,i_2,\ldots,i_k)\in{\mathbb
{Z}}^{k}}^{'}C^{\lambda}(i_1,i_2,
\ldots,i_k)\varepsilon _{n-i_1}\ldots\varepsilon_{n-i_k},
\end{equation}
where the prime $'$ indicates exclusion of the diagonals $i_{p}=i_q$,
$p\neq q$, $\{\varepsilon_n\}$ is as before, and
%
\begin{equation}
\label{eq:Cdef} C^{\lambda}(i_1,i_2,
\ldots,i_k)=\prod_{j=1}^{k}(i_j)^{d-1}_{+}e^{-\lambda(i_j)_{+}}
\end{equation}
for $d\in(\frac{1}{2}-\frac{1}{2k},\infty)$ and $\lambda>0$. Now, consider
\begin{equation*}
S^{\lambda}_{N}(t)=\sum_{n=1}^{[Nt]}Y^{\lambda,k}(n),
\quad 0\leq t\leq1.
\end{equation*}
\begin{thm}\label{thm:fdd theorem}
Let $Y^{\lambda,k}(n)$ be the discrete chaos process given by
\eqref{eq:chaostempereddiscrete}. Then
%
\begin{equation}
\frac{1}{N^{H}}S^{\frac{\lambda}{N}}_{N}(t)\Rightarrow
Z^{k}_{H,\lambda}(t),
\end{equation}
where $\Rightarrow$ means weak convergence in the Skorokhod space
$D[0,1]$ with uniform metric, $Z^{k}_{H,\lambda}(t)$ is the tempered
Hermite process in \eqref{eq:THPdefn}, and $H=1+kd-\frac{k}{2}$.
\end{thm}
\begin{rem}
The Lamperti's theorem \cite{lamperti} states that if
\begin{equation*}
\frac{1}{d(N)}\sum_{k=1}^{[Nt]}Y_{k}
\longrightarrowfdd Z(t)
\end{equation*}
and $d(N)\to\infty$ as $N\to\infty$, where $\{Y_k\}$ is stationary,
then $\{Z(t)\}_{t\geq0}$ is self-similar with stationary increments
($\longrightarrowfdd$ means the convergence of finite-dimensional
distributions). In our case, since the stationary processes ${\{
Y^{\frac{\lambda}{N}}_{k}\}}$ depend on $N$ through the parameter
$\lambda$, the limit process $\{Z^{k}_{H,\lambda}(t)\}_{t\geq0}$
need not be a self-similar process. Therefore, the result of Theorem
\ref{thm:fdd theorem} does not contradict the Lamperti theorem.
\end{rem}

\begin{proof}[Proof of Theorem~\ref{thm:fdd theorem}]
First, we show that
\begin{align}
\label{eq:fddgeneralchaos} %
\frac{1}{N^H}S^{\frac{\lambda}{N}}_{N}(t)&= \frac{1}{N^{H}}\sum_{n=1}^{[Nt]}Y^{\frac{\lambda}{N},k}(n)\notag\\
&=\sum_{(i_1,\ldots,i_k)\in\mathbb{Z}^{k}}\frac{1}{N^H}\sum_{n=1}^{[Nt]}C^{\frac{\lambda}{N}}(n-i_1,\ldots,n-i_k)\varepsilon _{i_1}\ldots\varepsilon_{i_k}\notag\\
&=Q_{k}(h_{t,N})\longrightarrowfdd Z^{k}_{H,\lambda}(t)\quad\text {as}\ N\to\infty,
\end{align}
where
\begin{equation*}
h_{t,N}(i_1,\ldots,i_k):=\frac{1}{N^H}
\sum_{n=1}^{[Nt]}C^{\frac
{\lambda}{N}}(n-i_1,
\ldots,n-i_k),
\end{equation*}
and $Q_{k}(\cdot)$ is defined by \eqref{eq:Qdefinition}. In order to
show \eqref{eq:fddgeneralchaos}, we just need to check that
%
\begin{equation}
\label{eq:firstpartoftheproof}
\big\|{\tilde{h}}_{t,N}(y_1,\ldots,y_k)-h_{t}(y_1,\ldots,y_k)\big\| _{L^{2}({\mathbb{R}}^k)}\to0
\end{equation}
as $N\to\infty$, where
\begin{equation*}
\begin{split} {\tilde{h}}_{t,N}(y_1,
\ldots,y_k)&:=N^{\frac
{k}{2}}h_{t,N}\bigl([Ny_1]+1,
\ldots,[Ny_k]+1\bigr)
\\
&=\frac{N^{\frac{k}{2}}}{N^H}\sum_{n=1}^{[Nt]}C^{\frac{\lambda
}{N}}
\bigl(n-[Ny_{1}]-1,\ldots,n-[Ny_{k}]-1\bigr), \end{split}
\end{equation*}
and $h_{t}(y_1,\ldots,y_k)$ is given by \eqref{eq:integrand}. Write
\begin{equation*}
\begin{split} {\tilde{h}}_{t,N}(y_1,
\ldots,y_k)&=\frac{N^{\frac{k}{2}}}{N^H}\sum_{n=1}^{[Nt]}C^{\frac{\lambda}{N}}
\bigl(n-[Ny_{1}]-1,\ldots ,n-[Ny_{k}]-1\bigr)
\\
&=\frac{1}{N^{1+kd-k}}\sum_{n=1}^{[Nt]}\prod
_{i=1}^{k}\bigl(n-[Ny_{i}]-1
\bigr)_{+}^{d-1}e^{-\frac{\lambda
}{N}(n-[Ny_{i}]-1)_{+}}
\\
&=\frac{1}{N}\sum_{n=1}^{[Nt]}\prod
_{i=1}^{k} \biggl(\frac
{n-[Ny_{i}]-1}{N}
\biggr)_{+}^{d-1}e^{-\lambda (\frac
{n-[Ny_{i}]-1}{N} )_{+}}
\\
&=\int_{0}^{t}\prod_{i=1}^{k}
\biggl(\frac{[Ns]-[Ny_i]}{N} \biggr)_{+}^{d-1}e^{-\lambda (\frac{[Ns]-[Ny_i]}{N} )_{+}}\,ds. \end{split} %
\end{equation*}
Let $d=1$. In this case,
\begin{equation*}
\biggl(\frac{[Ns]-[Ny]}{N} \biggr)^{d-1}_{+}e^{-\lambda (\frac
{[Ns]-[Ny]}{N} )_{+}}=e^{-\lambda (\frac{[Ns]-[Ny]}{N} )_{+}}
\leq e^{-\lambda(s-y)_{+}}e^{\frac{\lambda}{N}}
\end{equation*}
for all $N\geq1$, and hence
\begin{equation*}
\begin{split}
\Bigg|\prod_{i=1}^{k}e^{-\lambda (\frac{[Ns]-[Ny_i]}{N})_{+}}\Bigg|&\leq e^{\lambda k} \prod_{i=1}^{k}e^{-\lambda(s-y_i)_{+}}\\
&=:g_{1}(s-y_1,\ldots,s-y_{k}). \end{split}
\end{equation*}
Next, consider $0<d<1$. Since $[Ns]-[Ny]>Ns-Ny-1$, we get
\begin{equation*}
\biggl(\frac{[Ns]-[Ny]}{N} \biggr)^{d-1}_{+}< \biggl(
\frac
{Ns-Ny-1}{N} \biggr)^{d-1}_{+} \leq(s-y-1)^{d-1}_{+}
\end{equation*}
for all $N\geq1$, and hence
\begin{equation*}
\begin{split} \Bigg|\prod_{i=1}^{k}
\biggl(\frac{[Ns]-[Ny_i]}{N} \biggr)^{d-1}_{+}e^{-\lambda (\frac{[Ns]-[Ny_i]}{N} )_{+}}
\Bigg|&< \prod_{i=1}^{k}(s-y_{i}-1)^{d-1}_{+}e^{-\lambda(s-y_{i}-1)_{+}}
\\
&=:g_{2}(s-y_1,\ldots,s-y_{k}). \end{split}
\end{equation*}
Finally, suppose that $d>1$. Since $[Ns]-[Ny]<Ns-Ny+1$, we get
\begin{equation*}
\biggl(\frac{[Ns]-[Ny]}{N} \biggr)^{d-1}_{+}< \biggl(
\frac
{Ns-Ny+1}{N} \biggr)^{d-1}_{+} \leq(s-y+1)^{d-1}_{+}
\end{equation*}
for all $N\geq1$, and hence
\begin{equation*}
\begin{split} \Bigg|\prod_{i=1}^{k}
\biggl(\frac{[Ns]-[Ny_i]}{N} \biggr)^{d-1}_{+}e^{-\lambda (\frac{[Ns]-[Ny_i]}{N} )_{+}}
\Bigg|&< \prod_{i=1}^{k}(s-y_{i}+1)^{d-1}_{+}e^{-\lambda(s-y_{i}-1)_{+}}
\\
&=:g_{3}(s-y_1,\ldots,s-y_{k}). \end{split}
\end{equation*}
By the similar argument of Lemma~\ref{lem:g squar integrable}, we can
verify that
\begin{equation*}
\int_{\mathbb{R}^k}^{'} \Biggl(\int_{0}^{t}g_{i}(s-y_1,
\ldots ,s-y_{k})\,ds \Biggr)^2\,dy_{1},
\ldots,dy_{k}<\infty
\end{equation*}
for $i=1,2,3$. On the other hand, since $C^{\lambda}(i_1,\ldots
,i_{k})$ is continuous a.e.,\break $C^{\lambda}(\frac
{[Ns]-[Ny_1]}{N},\ldots,\frac{[Ns]-[Ny_k]}{N})$ converges a.e. to
$C^{\lambda}(s-y_1,\ldots,s-y_k)$ as $N\to\infty$. Now apply the
dominated convergence theorem to get the desired result~\eqref
{eq:firstpartoftheproof}.

In order to show the tightness, we need to verify that
%
\begin{equation}
\label{eq:tightnesscondition}
\mathbb{E} \big|N^{-H} \bigl(S^{\frac{\lambda}{N}}_{N}(t)-S^{\frac
{\lambda}{N}}_{N}(s)
\bigr) \big|^{2\gamma}\leq C\big|F_{n}(t)-F_{n}(s)\big|^{2\alpha},
\quad 0\leq s<t\leq1,
\end{equation}
where $\gamma>0$ and $\alpha>\frac{1}{2}$ (here $\{F_{n}\}_{n\geq
1}$ is a sequence of nondecreasing continuous functions on $[0,1]$ that
are uniformly bounded and satisfy
\begin{equation*}
\lim_{\delta\to0}\limsup_{n\to\infty}
\omega_{\delta}(F_n)=0,
\end{equation*}
where $\omega_{\delta}(F):=\sup_{|t-s|<\delta}|F(t)-F(s)|$ for
$\delta>0$). See Lemma~4.4.1 in \cite{koul} for more details. Observe that
\begin{align*}
\label{eq:tightness1} %
S^{\frac{\lambda}{N}}_{N}(t)&=\sum_{n=1}^{[Nt]}Y^{\frac{\lambda}{N},k}(n)=\sum_{(i_1,\ldots,i_k)\in\mathbb{Z}}^{'}\sum_{n=1}^{[Nt]}C^{\frac{\lambda}{N},k}(n-i_1,\ldots,n-i_k)\varepsilon _{i_1}\ldots\varepsilon_{i_k}\\
&=\sum_{(i_1,\ldots,i_k)\in\mathbb{Z}}^{'}\sum_{n=1}^{[Nt]}\prod_{j=1}^{k}(n-i_j)^{d-1}_{+}e^{-\frac{\lambda}{N}(n-i_j)_{+}}\varepsilon_{i_1}\ldots\varepsilon_{i_k}\\
&=\sum_{(i_1,\ldots,i_k)\in\mathbb{Z}}^{'}N^{kd-k+1} \Biggl[\frac{1}{N}\sum_{n=1}^{[Nt]}\prod_{j=1}^{k} \biggl(\frac{n-i_j}{N}\biggr)^{d-1}_{+}e^{-\frac{\lambda}{N}(n-i_j)_{+}} \Biggr]\varepsilon_{i_1}\ldots\varepsilon_{i_k}\\
&=\sum_{(i_1,\ldots,i_k)\in\mathbb{Z}}^{'}N^{kd-k+1} \Biggl[\int_{0}^{t}\prod_{j=1}^{k}\biggl(\frac{[Ny]+1-i_j}{N} \biggr)^{d-1}_{+}\\
&\quad \times e^{-\frac{\lambda}{N}([Ny]+1-i_j)_{+}}\,dy\Biggr] \varepsilon_{i_1}\ldots\varepsilon_{i_k}\\
&=N\sum_{(i_1,\ldots,i_k)\in\mathbb{Z}}^{'} \Biggl[\int_{0}^{t}\prod_{j=1}^{k}\bigl([Ny]+1-i_j\bigr)^{d-1}_{+}e^{-\frac{\lambda}{N}([Ny]+1-i_j)_{+}}\,dy\Biggr] \varepsilon_{i_1}\ldots\varepsilon_{i_k}.
\end{align*}
Therefore, we get
\begin{align*}
&\mathbb{E} \big|N^{-H}\bigl(S^{\frac{\lambda}{N}}_{N}(t)\,{-}\,S^{\frac{\lambda}{N}}_{N}(s)\bigr) \big|^2\notag\\
&\quad=N^{2-2H}\\
 &\qquad\times\mathbb{E} \Bigg|\sum_{(i_1,\ldots,i_k)\in\mathbb{Z}}^{'} \Biggl[\int_{0}^{t-s}\prod_{j=1}^{k}\bigl([Ny]+1-i_j\bigr)^{d-1}_{+}e^{-\frac{\lambda}{N}([Ny]+1-i_j)_{+}}\,dy\Biggr] \varepsilon_{i_1}\ldots\varepsilon_{i_k}\Bigg|^{2}\\
&\quad\leq k!N^{2-2H}\sum_{(i_1,\ldots,i_k)\in\mathbb{Z}}^{'}\Biggl[\int_{0}^{t-s}\prod_{j=1}^{k}\bigl([Ny]+1-i_j\bigr)^{d-1}_{+}e^{-\frac{\lambda}{N}([Ny]+1-i_j)_{+}}\,dy \Biggr]^{2}\\
&\quad=k!N^{2-2H+k}\\
&\qquad \times\int_{\mathbb{R}^k}^{'} \Biggl[\int_{0}^{t-s}\prod_{j=1}^{k}\bigl([Ny]+1-[Nx_j]\bigr)^{d-1}_{+}e^{-\frac{\lambda}{N}([Ny]+1-[Nx_j])_{+}}\,dy \Biggr]^{2}\,dx_{1}\ldots dx_{k}.
\end{align*}
Now, we consider two different cases corresponding with the range of
$d$. First, assume that $\frac{1}{2}-\frac{1}{2k}<d\leq1$
(equivalently, $\frac{1}{2}<H\leq1+\frac{k}{2}$):
Since $[Ny]-[Nx_j]+1>Ny-Nx_{j}$, we can write
\begin{align*}
&\mathbb{E} \big|N^{-H}\bigl(S^{\frac{\lambda}{N}}_{N}(t)-S^{\frac{\lambda}{N}}_{N}(s)\bigr) \big|^2\notag\\
&\quad\leq k!N^{2-2H+k}\\
&\qquad \times\int_{\mathbb{R}^k}^{'} \Biggl[\int_{0}^{t-s}\prod_{j=1}^{k}\bigl([Ny]+1-[Nx_j]\bigr)^{d-1}_{+}e^{-\frac{\lambda}{N}([Ny]+1-[Nx_j])_{+}}\,dy \Biggr]^{2}\,dx_{1}\ldots dx_{k}\\
&\quad \leq k!N^{2-2H+2kd-k}\int_{\mathbb{R}^k}^{'} \Biggl[\int_{0}^{t-s}\prod_{j=1}^{k}(y-x_j)^{d-1}_{+}e^{-\lambda(y-x_j)_{+}}\,dy\Biggr]^{2} \,dx_{1}\ldots dx_{k}\\
&\quad =k!\int_{\mathbb{R}^k}^{'} \Biggl[\int_{0}^{t-s}\prod_{j=1}^{k}(y-x_j)^{d-1}_{+}e^{-\lambda(y-x_j)_{+}}\,dy\Biggr]^{2} \,dx_{1}\ldots dx_{k}.
\end{align*}
Now, let $d>1$. Since $Ny-Nx_{j}<[Ny]-[Nx_j]+1<Ny-Nx_{j}+N$, we have
\begin{align*}
&\mathbb{E} \big|N^{-H}\bigl(S^{\frac{\lambda}{N}}_{N}(t)-S^{\frac{\lambda}{N}}_{N}(s)\bigr) \big|^2\notag\\
&\quad\leq k!N^{2-2H+k}\\
&\qquad \times\int_{\mathbb{R}^k}^{'} \Biggl[\int_{0}^{t-s}\prod_{j=1}^{k}\bigl([Ny]+1-[Nx_j]\bigr)^{d-1}_{+}e^{-\frac{\lambda}{N}([Ny]+1-[Nx_j])_{+}}\,dy \Biggr]^{2}\,dx_{1}\ldots dx_{k}\\
&\quad \leq k!\int_{\mathbb{R}^k}^{'} \Biggl[\int_{0}^{t-s}\prod_{j=1}^{k}(y-x_j+1)^{d-1}_{+}e^{-\lambda(y-x_j)_{+}}\,dy\Biggr]^{2} \,dx_{1}\ldots dx_{k}\\
&\quad =k!\int_{\mathbb{R}^k}^{'} \Biggl[\int_{0}^{t-s}\prod_{j=1}^{k}(y-x_j+1)^{d-1}e^{-\lambda(y-x_j+1)}e^{\lambda}\mathbf {1}_{\{y>x_j\}}\,dy \Biggr]^{2} \,dx_{1}\ldots dx_{k}\\
&\quad =e^{2\lambda k}k!\int_{\mathbb{R}^k}^{'} \Biggl[\int_{0}^{t-s}\prod_{j=1}^{k}(y-z_j)^{d-1}e^{-\lambda(y-z_j)}\mathbf{1}_{\{y>z_j+1\}}\,dy \Biggr]^{2} \,dz_{1}\ldots dz_{k}\\
&\quad \leq e^{2\lambda k}k!\int_{\mathbb{R}^k}^{'} \Biggl[\int_{0}^{t-s}\prod_{j=1}^{k}(y-z_j)^{d-1}_{+}e^{-\lambda(y-z_j)_{+}}\,dy\Biggr]^{2} \,dz_{1}\ldots dz_{k}.
\end{align*}
Therefore,
\begin{align*}
&\mathbb{E} \big|N^{-H} \bigl(S^{\frac{\lambda}{N}}_{N}(t)-S^{\frac{\lambda}{N}}_{N}(s)\bigr) \big|^2\\
&\quad \leq e^{2\lambda k}k!\int_{\mathbb{R}^k}^{'}\Biggl[\int_{0}^{t-s}\prod_{j=1}^{k}(y-z_j)^{d-1}_{+}e^{-\lambda(y-z_j)_{+}}\,dy\Biggr]^{2} \,dz_{1}\ldots dz_{k}
\end{align*}
for any $d>\frac{1}{2}-\frac{1}{2k}$ (equivalently, $H>\frac
{1}{2}$). According to the proof of Lemma~\ref{lem:g squar integrable},
\begin{equation*}
e^{2\lambda k}k! \int_{\mathbb{R}^k}^{'} \Biggl[\int
_{0}^{t-s}\prod_{j=1}^{k}(y-z_j)^{d-1}_{+}e^{-\lambda(y-z_j)_{+}}
\,dy \Biggr]^{2}\, dz_{1}\ldots dz_{k}\leq
C|t-s|^{2H}
\end{equation*}
for $\frac{1}{2}<H<1$ and
\begin{equation*}
e^{2\lambda k}k! \int_{\mathbb{R}^k}^{'} \Biggl[\int
_{0}^{t-s}\prod_{j=1}^{k}(y-z_j)^{d-1}_{+}e^{-\lambda(y-z_j)_{+}}
\,dy \Biggr]^{2} \,dz_{1}\ldots dz_{k}\leq
C|t-s|^{2}
\end{equation*}
for $H>1$. Now, it remains to apply \eqref{eq:tightnesscondition} by
selecting $\gamma=1$, $\alpha=\min\{H,1\}$, and $F_{n}(t)=t$ to get
the desired result.
\end{proof}

\section*{Acknowledgments}
The author would like to thank Hira Koul and Mark Meerschaert from
Michigan State University for fruitful discussions.


%
\end{document}